\documentclass{article}
\usepackage[utf8]{inputenc}
\usepackage{amsthm, amssymb, amsmath}
\usepackage{xcolor}
\usepackage{authblk}

\title{Conflict-free coloring on closed neighborhoods of bounded degree graphs}
\author{Sriram Bhyravarapu}
\author{Subrahmanyam Kalyanasundaram}
\author{Rogers Mathew}
\affil
{
Department of Computer Science and Engineering, \authorcr
Indian Institute of Technology Hyderabad, India - 502285. \authorcr {\tt \{cs16resch11001, subruk, rogers\}@iith.ac.in}
}

\theoremstyle{definition}

\theoremstyle{plain}
\newtheorem{theorem}{Theorem}

\newtheorem{observation}[theorem]{Observation}

\newtheorem{definition}[theorem]{Definition}

\theoremstyle{remark}

\begin{document}

\maketitle
\begin{abstract}
The \emph{closed neighborhood conflict-free chromatic number} of a graph $G$, denoted by $\chi_{CN}(G)$, is the minimum number of colors required to color the vertices of $G$ such that for every vertex, there is a color that appears exactly once in its closed neighborhood. 
Pach and Tardos [Combin. Probab. Comput. 2009] showed that $\chi_{CN}(G) = O(\log^{2+\varepsilon} \Delta)$, for any $\varepsilon > 0$,
where $\Delta$ is the maximum degree. In [Combin. Probab. Comput. 2014], Glebov, Szab\'o and Tardos showed existence of graphs $G$ with $\chi_{CN}(G) = \Omega(\log^2\Delta)$. In this paper, we bridge the gap between the two bounds by showing that $\chi_{CN}(G) = O(\log^2 \Delta)$. 
\end{abstract}
Conflict-free coloring was introduced in 2003 \cite{Even2002} motivated by problems arising from situations in wireless communication. Over the past two 
decades, conflict-free coloring has  been extensively studied \cite{smorosurvey}.
\begin{definition}[Conflict-free chromatic number of hypergraphs]\label{def:CF_hypergraph}
The \emph{conflict-free chromatic number} of a hypergraph $H=(V,E)$  is the minimum number of colors required to color the points in $V$ such that every $e \in E$ contains a point whose color is distinct from that of every other point in $e$. 
\end{definition}
\noindent Conflict-free coloring has also been studied in the context of hypergraphs created out of simple graphs. Two such variants are \emph{conflict-free coloring on closed neighborhoods} and 
\emph{conflict free coloring on open neighborhoods}. In this manuscript, we focus on the former variant. 
Given a graph $G$, for any vertex $v \in V(G)$, let $N_G(v) : = \{u \in V(G):\{u,v\} \in E(G)\}$ denote the \emph{open neighborhood} of $v$ in $G$. Let $N_G[v] := N_G(v) \cup \{v\}$ denote the \emph{closed neighborhood} of $v$ in $G$. 
\begin{definition}[{Closed neighborhood conflict-free chromatic number}]\label{defn:closed_CF}
Given a graph $G = (V, E)$, a conflict-free coloring on closed neighborhoods (CFCN coloring) 
is an assignment of colors
$C: V(G) \rightarrow \{1, 2, \ldots, k\}$ such that for every $v \in V(G)$, there exists an 
$i \in \{1, 2, \ldots, k\}$ such that $|N[v] \cap C^{-1} (i)| = 1$. The smallest $k$ required for such a coloring 
is called the CFCN chromatic number of $G$, denoted $\chi_{CN}(G)$.
\end{definition} 
In other words, given a graph $G$, let $H$ be the hypergraph with $V(H) = V(G)$ and $E(H) = \{N_G[v]:v \in V(G)\}$. Then,   $\chi_{CN}(G)$ is equal to the conflict-free chromatic number of the hypergraph $H$ created from $G$. 
Pach and Tardos \cite{Pach2009} showed that for a graph $G$ with maximum degree
$\Delta$, the CFCN chromatic number $\chi_{CN}(G) = O(\log^{2 + \varepsilon} \Delta)$ for any $\varepsilon > 0$. We improve this bound and show the following.
\begin{theorem}\label{thm:cfcntight}
Let $G$ be a graph with maximum degree $\Delta$. Then $\chi_{CN}(G) = O(\log^2 \Delta)$.
\end{theorem}
In 2014, Glebov, Szab\'o and Tardos \cite{glebov2014conflict} showed the existence of graphs $G$
on $n$ vertices
such that $\chi_{CN}(G) = \Omega(\log^2 n)$. Since  
$\Delta < n$, our bound in Theorem \ref{thm:cfcntight} is tight up to constants. 

Before we proceed to the proof, we explain some notations. All logarithms we consider 
here are to the base $2$. 
Given a graph $G$ and a set $S \subseteq V(G)$, we use $G[S]$ to denote the subgraph of $G$ induced on the vertex set $S$. For any two vertices $u,v \in V(G)$, we use $dist_G(u,v)$ to denote the number of edges in
a shortest path between $u$ and $v$ in $G$. We set $dist_G(u,v) = \infty$ 
when there is no path between $u$ and $v$ in $G$.
\begin{definition}[Maximal Distance-$3^+$ Set]
For a graph $G$, a \emph{maximal distance-$3^+$ set} is a set $A \subseteq V(G)$ that satisfies the following:
\begin{enumerate}
    \item For every two distinct $u, v \in A$, $dist_G(u,v) \geq 3$. 
    \item For every $v \in V(G)\setminus A$, $\exists u \in A$ such that $dist_G(u,v) < 3$. 
\end{enumerate}
\end{definition}
Let $A$ be a maximal distance-$3^+$ set in $G$. Let $B = \{v \in V(G) \setminus A:v\mbox{ has a neighbor in }A\}$ and let $C = V(G) \setminus (A
\cup B)$. 
We make the following observations. 
\begin{observation}
\label{obv:dist-3-set_exact_1_neighbor}
Every vertex in $B$ has exactly one neighbor in $A$.  
\end{observation}
\begin{observation}
\label{obv:dist-3-set_atleast_1_neighbor}
Every vertex in $C$ has at least one neighbor in $B$.  
\end{observation}
Our proof uses the following theorem on conflict-free coloring on 
hypergraphs due to  Pach and Tardos \cite{Pach2009}.
\begin{theorem}[Theorem 1.2 in \cite{Pach2009}]\label{thm_pach_main}
For any positive integers $t$ and $\Gamma$, the conflict-free chromatic number of any hypergraph in which each edge is of size at least $2t-1$ and each edge intersects at most $\Gamma$ others is $O(t\Gamma^{1/t}\log \Gamma)$. There is a randomized polynomial time algorithm to find such a coloring.   
\end{theorem}

\begin{proof}[Proof of Theorem \ref{thm:cfcntight}]
We perform the following iterative process starting with $G_0 = G$.

\noindent \textbf{Iterative coloring process:}
Let $A_i$ be a maximal distance-$3^+$ set in $G_i$. Let $B_i := \{v \in V(G_i) \setminus A_i: v \mbox{ has a neighbor in }A_i\}$ and $C_i := V(G_i) 
\setminus (A_i \cup B_i)$.  Assign a  color $c_i$ to all the vertices in $A_i$.  Observation \ref{obv:dist-3-set_exact_1_neighbor} combined  with 
the fact that $A_i$ is an independent set in $G_i$ imply that for every vertex  $v \in A_i \cup B_i$,
$N_G[v]$ contains exactly one vertex with the color $c_i$.
Repeat the above process with $G_{i+1} = G[C_i]$. 

The iterative process is repeated till one of the following two conditions is satisfied: 
(i) $G_{i}$ is the empty graph, or (ii) $i = k = 4 \log \Delta$. 

If the process terminated with $i <  4\log \Delta$, then we have CFCN-colored  $G$ with $O(\log \Delta)$ colors. Suppose it terminated 
with $i = k= 4\log \Delta$. We know that every vertex in $V(G) \setminus C_k$ has some color appearing exactly once in its closed 
neighborhood under the present coloring. In order to complete the proof, we need to extend this `nice' property to the vertices of $C_k$ as well. If $C_k$ is the empty set, then the proof is complete. 
Asssume $C_k$ is non-empty. Let $H$ be a hypergraph constructed from $G$  as explained here. Let $V(H) = B_0 \cup B_1 \cup \cdots \cup B_k$ 
and $E(H) = \{e_v : v \in C_k\}$, where 
$e_v = \{N_G(v) \cap V(H)\}$. 
We make a crucial observation here that each vertex in the set $\cup_{i=0}^k B_i$ is uncolored so far. Consider a vertex $v$ in $C_k$. From Observation \ref{obv:dist-3-set_atleast_1_neighbor}, $v$ has at least one neighbor in each of $B_0, B_1, \ldots , B_k$ thus making the size of $e_v$ at least $4\log \Delta + 1$. Further, each hyperedge in this hypergraph intersects at most $\Delta^2$ other hyperedges. Substituting $t = 2\log \Delta$ and $\Gamma = \Delta^2$ in Theorem \ref{thm_pach_main}, we can see that the conflict-free chromatic number of the hypergraph $H$ is $O(\log^2\Delta)$. We ensure that the colors we use to color the points in the hypergraph $H$  is disjoint from the set $\{c_0, c_1, \ldots , c_k\}$. Now, consider the graph $G$. For each vertex $v \in B_0 \cup B_1 \cup \cdots \cup B_k = 
V(H)$, we assign the color it obtained while coloring $H$. This would mean that every vertex in $C_k$ now has some color appearing exactly once in its closed neighborhood in $G$ and thereby satisfying the `nice' property mentioned above. Finally, use a new color (that has not been used so far) to color all the so far uncolored vertices in $G$. 
This completes  the proof
of the theorem.   
\end{proof}

\bibliographystyle{plain}
\bibliography{bibfile}

\end{document}